\documentclass[12pt]{article}
\usepackage[utf8]{inputenc}
\usepackage{tikz}
\usepackage{amsthm}
\usepackage{amsmath}
\usepackage{amssymb}
\usepackage{amsthm}
\usepackage{amscd}
\usepackage[all]{xy}
\usepackage[margin=1in]{geometry}
\usepackage{comment}
\usepackage{pdfpages}
\usepackage{tikz-cd}
\usepackage[T1]{fontenc}
\usepackage{amsmath,amsfonts,amssymb,enumerate}
\usepackage{amscd}
\usepackage{graphicx}
\usepackage[T1]{fontenc}
\usepackage{amsmath,amsfonts,amssymb,enumerate}
\usepackage{amsthm} 
\usepackage[english]{babel}

\newtheorem{theorem}{Theorem}[section]
\newtheorem{lemma}[theorem]{Lemma}
\newtheorem{proposition}[theorem]{Proposition}
\newtheorem{corollary}[theorem]{Corollary}

\newtheorem{construction}[theorem]{Construction}
\newtheorem{prop}[theorem]{Proposition}

\theoremstyle{definition}

\newtheorem{definition}[theorem]{Definition}
\newtheorem{remark}[theorem]{Remark}
\newtheorem{example}[theorem]{Example}

\theoremstyle{definition}
\numberwithin{equation}{section}

\newcommand{\N}{\mathbb{N}}

\newcommand{\ord}{\mathrm {ord}}

\newcommand{\m}{\mathfrak{m}}

\newcommand{\hgt}{\mathrm {ht}}

\def\m{{\mathfrak m}} 

\def\p{{\mathfrak p}}    
\def\q{{\mathfrak q}}    

\def\Spec{\mbox{\rm Spec }}

\def\hgt{\mbox{\rm ht }}

\def\dim{\mbox{\rm dim}}

\def\ord{\mbox{\rm ord}}
\def\trdeg{\mbox{\rm tr.deg}}

\begin{document}

\title{Directed unions of local monoidal transforms and GCD domains}

\author{ Lorenzo Guerrieri
\thanks{Universit\`a di Catania, Dipartimento di Matematica e
Informatica, Viale A. Doria, 6, 95125 Catania, Italy \textbf{Email address:} guelor@guelan.com }} 

\maketitle






\begin{abstract}
\noindent Let $(R, \mathfrak{m})$ be a regular local ring of dimension $d \geq 2$. 
A local monoidal transform of $R$ is a ring of the form $R_1= R[\frac{\mathfrak{p}}{x}]_{\mathfrak{m}_1}$ where $x \in \p$ is a regular parameter, $\p$ is a regular prime ideal of $R$ and $ \mathfrak{m}_1 $ is a maximal ideal of $ R[\frac{\mathfrak{p}}{x}] $ lying over $ \mathfrak{m}. $ \\
In this article we study some features of the rings $ S= \cup_{n \geq 0}^{\infty} R_n $ obtained as infinite directed union of iterated local monoidal transforms of $R$. \\
In order to study when these rings are GCD domains, we also provide results in the more general setting of directed unions of GCD domains.

\medskip

\end{abstract}

\section{Introduction}  
 Let $(R,\m)$ be a regular local ring.  An element $x \in \m$ is a {\it regular parameter} if $x \not\in \m^2$.   
  A prime ideal $\p$ of $R$ is a regular prime if $\frac{R}{\p}$ is again a regular local ring (i.e. $\p$ is generated by regular parameters).
 Let $\p$ be a  regular prime ideal of $R$
  with $\hgt \p \ge 2$. 
  A \textit{local monoidal transform} of $R$ is an overring of the form $$R_1=R \left[ \frac{\p}{x} \right]_{\m_1}$$ where $x \in \p \setminus \p^2$ and $\m_1$ is a maximal ideal of $R[\frac{\p}{x}]$ lying over $\m$. 
  
  In the particular case in which $\p$ is equal to the maximal ideal $\m$, the ring $R_1$ is called a \textit{local quadratic transform} of $R$. These concepts were mainly studied starting from the 50's by authors like Abhyankar \cite{Abh} \cite{Abh3}, Hironaka \cite{hironaka} and Zariski \cite{ZarR} for their geometric interpretation, useful to solve problems related to resolution of singularities of algebraic varieties.
  Sequences of iterated local quadratic transform of rings of the same dimension are of great interest, since this process corresponds
to the geometric notion of following a non singular closed point of an algebraic variety through repeated
blow-ups. In his paper \cite{Sha}, David Shannon discussed several properties of sequences of local monoidal transforms focusing on the parallelism between the algebraic point of view of regular local rings and its global geometric interpretation of projective models.
  
  Given a valuation overring $V$ of $R$, there exists a unique sequence of local quadratic transform of $R$ such that $$ R \subseteq R_1 \subseteq R_2 \subseteq ... \subseteq V$$ and $R_n$ is a local quadratic transform of $R_{n-1}$ for every $n$. This sequence is called \textit{sequence of local quadratic transform along $V$}.
  In this case the first local quadratic transform $ R_1= R \left[ \frac{\p}{x} \right]_{\m_1}$ is such that the value of $x$ with respect to the valuation induced by $V$ is minimal in $\m$ and $ \m_1 $ is the center of $V$ over the ring $ R \left[ \frac{\p}{x} \right] $. The others transforms $R_n$ are built in the same way.
  
  
Assuming $\dim R_n \ge 2$, we have  $R_n \subsetneq  R_{n+1}$ for each positive integer $n$ 
        and $\bigcup _n R_n$ is an infinite ascending union.
        If $\dim R = 2$, the union $\bigcup _n R_n$ is equal to a valuation overring of $R$ \cite[Lemma 12]{Abh}.  
If $\dim R > 2$, then,  examples due to Shannon \cite{Sha} show
that  the union may be not a valuation ring. 
    This fact leaded to the study of the structure of the rings of the form $S= \bigcup _n R_n$, which are called \textit{quadratic Shannon extension of $R$} \cite{HLOST}.
  
Recently, authors like Heinzer, Loper, Olberding, Schoutens, Toeniskoetter (\cite{HLOST}, \cite{HOT}, \cite{GHOT}) and others studied the ideal-theoretic structure of quadratic Shannon extensions without focusing particularly on the geometric origin of these concepts. The tools used in these works are usually those from multiplicative ideal theory.  In \cite{HLOST} and \cite{HOT}, the authors call a quadratic Shannon extension $S$ simply a Shannon extension of $R$. 
     
In this paper we distinguish this class of rings, that we call quadratic Shannon extensions, from the larger class of rings of \textit{monoidal Shannon extensions} formed by the rings of the form $S= \bigcup _n R_n$ where $R_n$ are iterated monoidal transform of $R$. The aim of this paper is to study the ring-theoretic structure of some classes of monoidal Shannon extensions comparing their properties with those of the quadratic extensions.

For instance, while there are many valuation overrings that birationally dominates $R$ which are not quadratic Shannon extension (for instance the unique sequence of local quadratic transform of $R$ along a valuation overring of rank at least $3$ yields to a Shannon extension which is not a valuation ring \cite[Theorem 8.1]{HLOST}), under certain hypothesis, easily fulfilled by regular local rings arising in a geometric context, every valuation overring birationally dominating $R$ is a monoidal Shannon extension.
This is one of the most important result of Shannon's paper \cite{Sha} and we state it for completeness. 

\begin{theorem} \label{shteo}  Let $(R, \m)$ be an excellent regular local ring of dimension greater than one such that one of the following conditions hold:
\begin{enumerate}
\item The residue field $\frac{R}{\m}$ has characteristic zero (in any dimension of $R$).
\item $R$ is equicharacteristic of dimension at most $3$.
\end{enumerate}
 Then every valuation ring $V$ that birationally dominates $R$ is a union of local monoidal transform of $R$.
   \end{theorem}
  
 
 
 
 Hence, in a ring theoretic language, assuming mild geometric hypothesis on $R$, we can say that every valuation overring that birationally dominates $R$ is a monoidal Shannon extension of $R$. The converse is clearly not true, as Shannon itself mentioned providing two examples. 

 An interesting and wide open problem is now to study with ideal-theoretic tools the monoidal Shannon extension which are neither quadratic extensions nor valuation rings. 
 In this work we provide examples of classes of monoidal non-quadratic Shannon extensions and we study their properties. 
  In the study of monoidal Shannon extensions we are interested in establish under which assumptions they are GCD domains. We recall that an integral domain $D$ is a GCD domain if  for all $a, b \in D$, $aD \cap bD$ is a principal ideal of $D$. The survey paper of Dan Anderson \cite{Anderson} and the book of Robert Gilmer \cite{Gil} are good references for general theory of GCD domains. 
 
 In \cite[Theorem 6.2]{GHOT} is proved that a quadratic Shannon extension of a regular local ring is a GCD domain if and only if it is a valuation ring. In the monoidal case we find Shannon extensions which are GCD domains but not valuation domains.
 
 A useful approach for this question is to study in general directed unions of GCD domains and see under which conditions, a finite intersection of principal ideals of the union is principal. We dedicate Section 2 to this question, introducing two classes of domains, that may seem interesting also in a general context: the first one is formed by the domains in which finite intersections of principal ideals are finitely generated only if are principal. Pre-Schreier domains and hence directed unions of GCD domains fulfill this property \cite{zafrullah}. The second class is formed by the domains having a finite intersection of principal ideals finitely generated only when equal to one of the principal ideals intersected (see Definitions \ref{d1} and \ref{d2}). We characterize when a directed union of Noetherian GCD domains is in this second class, deriving as a consequence a generalization of \cite[Theorem 6.2]{GHOT}.
 
 In Section 3, we specialize to the study of monoidal Shannon extensions. We characterize when they are Noetherian and we give a detailed description of a specific example of a monoidal non-quadratic Shannon extension. This ring, defined in Construction \ref{construction} it is not a valuation domain and it will be shown at the end of this paper to be a GCD domain.
 
 In Section 4, we introduce the concept of \it chain-prime ideal \rm of a monoidal Shannon extension $S = \bigcup_{n \in \N}R_n$. They are "special" prime ideals of $S$ obtained as infinite ascending union of collections of prime ideals $\p_n \subseteq R_n$, that are the prime ideals defining the different iterated monoidal transforms. We also provide bounds on the height of these special prime ideals in term of the height of the ideals $\p_n$ in $R_n$ and we connect this bounds to possible bounds on the dimension of valuation overrings obtained as Shannon extensions.
 
 In Section 5, we use the theory developed in Section 4 to study the monoidal Shannon extensions having principal maximal ideal and we describe their features using their representations as pullback diagrams.

 Finally, in Section \ref{ultima} we characterize the GCD domains among the large class of the monoidal Shannon extensions having only finitely many chain-prime ideals. We prove that, in this case, a monoidal Shannon extension is a GCD domain if and only if the localizations at all the chain-prime ideals are valuation rings.

 Our standard notation is as in Matsumura \cite{Mat}.  Thus a local ring need not be Noetherian.
We call an extension ring $B$ of an integral domain $A$ an {\it overring of} $A$ if $B$ is a subring of the quotient field of $A$.  If, in addition,  $A$ and $B$ are local and the inclusion map $A \hookrightarrow B$ is a 
local homomorphism (i.e. the maximal ideal of $A$ is contained in the maximal ideal of $B$),  we say that $B$ {\it birationally dominates}  $A$. 
If $P$ is a prime ideal of a ring $A$, we denote by $\kappa(P)$ the residue field of $A_P$.  

 \medskip
 

 \section{Directed unions of GCD domains}

An integral domain $D$ is a \it GCD domain \rm if for every $a, b \in D$, $aD \cap bD$ is a principal ideal of $D$ and it is a \it finite conductor \rm domain if for every $a, b \in D$, $aD \cap bD$ is finitely generated.
We give here a new definition whose name is inspired from the fact that in a Bezout domain, any finitely generated ideal is principal.

\begin{definition} 
\label{d1}
 An integral domain $D$ is said a Bezout Intersection domain (BID) if for any finite collection of elements $a_1, a_2, \ldots, a_s \in D$, the ideal $$I= a_1D \cap a_2D \cap \ldots \cap a_sD$$ is either principal or it is non finitely generated.
 \end{definition}

It is straightforward to observe that a BID is Noetherian or it is a finite conductor if and only if it is a GCD domain. Moreover the invertible ideals of a BID are principal and hence a BID it is Pr\"{u}fer domain if and only if it is Bezout domain.

\begin{theorem} 
\label{uniongcd}
Let $ \lbrace D_n \rbrace_{n \geq 1} $ be an infinite family of GCD domains such that $D_n \subseteq D_{n+1}$ for every $n$.
Let $D = \bigcup_{n=1}^{\infty} D_n$ be the directed union of the GCD domains $D_n$. Then $D$ is a BID.
\end{theorem}

\begin{proof}
Let $a_1, a_2, \ldots, a_s \in D$ and set $I= a_1D \cap a_2D \cap \ldots \cap a_sD$. We show that $$I= \bigcup_{n \geq N}( a_1D_n \cap a_2D_n \cap \ldots \cap a_sD_n)$$ for $N$ large enough such that $a_1, a_2, \ldots, a_s \in D_N$. The containment ''$\supseteq$'' is clear. For the other, take $d \in I$, then $d = a_1 d_1 = a_2d_2= \ldots = a_s d_s$ for $d_1, \ldots, d_s \in D$. It follows that $d \in a_1D_n \cap a_2D_n \cap \ldots \cap a_sD_n$ for any $n$ large enough such that $ a_i, d_i \in D_n $ for every $i=1, \ldots, s$. Now, since $D_n$ is a GCD domain, $ a_1D_n \cap a_2D_n \cap \ldots \cap a_sD_n = c_n D_n $ for some $c_n \in D_n$ ($c_n$ is the least common multiple in $D_n$ of $a_1, a_2, \ldots, a_s$). Since for every $n \geq N$, $c_n D_n \subseteq c_{n+1} D_{n+1}$, it is easy to see that $$I= \bigcup_{n \geq N}(c_nD_n) = \bigcup_{n \geq N}(c_nD). $$ Hence $I$ is an ascending union of principal ideals of $D$ and therefore it is either principal or it is not finitely generated.
\end{proof}

\begin{remark} 
\label{r1}
Theorem \ref{uniongcd} can be proved also observing that a directed union of GCD domains is a Schreier domain and applying Theorem 3.6 of \cite{zafrullah}. In particular a pre-Schreier domain is a BID. Anyway, we will need the argument of the elementary proof given for Theorem \ref{uniongcd} later in this article.
\end{remark}

\begin{definition} 
\label{d2}
 An integral domain $D$ is a Strong Bezout Intersection Domain (SBID) if, for $a_1, \ldots, a_s \in D$, whenever  the ideal $I= a_1D \cap a_2D \cap \ldots \cap a_sD$ is finitely generated, then $I=a_iD$ for some $i=1, \ldots, s.$
 \end{definition}

It follows easily from the definition that a SBID is a finite conductor domain or a GCD domain if and only if it is a valuation domain.

\begin{prop} 
\label{localSBID}
A SBID is local.
\end{prop}

\begin{proof}
Let $a,b$ be non-units in a SBID $D$. We only need to prove that $a+b$ is not a unit. But, if by way of contradiction we assume $a+b$ to be a unit, then $(a,b)D=D$ and hence $aD \cap bD = abD$ is principal. Thus either $abD=aD$ or $abD=bD$ which implies that either $a$ or $b$ is a unit.
\end{proof}

\begin{definition} 
\label{gcdn}
Let $D = \bigcup_{n=1}^{\infty} D_n$ be the directed union of some GCD domains $D_n$. For $a,b \in D$, call $ \gcd_n(a,b) $ a greatest common divisor of $a,b$ in the ring $D_n$ (it exists when $a,b \in D_n$). In the next we are going often to say $\gcd_n(a,b) = \gcd_m(a,b)$ for $n \leq m$, but meaning that they are associated in $D_m $ and not necessarily equal.
 \end{definition}

 It follows from the proof of Theorem \ref{uniongcd}, that in a directed union of GCD domains $D$, a finite intersection of principal ideals $I= a_1D \cap a_2D \cap \ldots \cap a_sD$ is principal if and only if the least common multiple of $a_1, a_2, \ldots, a_s$ in $D_n$ stabilizes for $n \gg 0.$ It is easy to check that this least common multiple in $D_n$ can be expressed via the Inclusion–Exclusion principle as a fraction in term of the product $ a_1a_2 \cdots a_s $ and all the terms of the form 
$\gcd_n(a_{i_1}, a_{i_2}, \ldots, a_{i_k})$ for all the possible combinations $i_j \in \lbrace 1,2, \ldots, s \rbrace$ and $2 \leq k \leq s.$ Hence, assuming that the integral domains $D_n$ are Noetherian (and thus UFDs), the least common multiple of $a_1, a_2, \ldots, a_s$ in $D_n$ stabilizes for $n \gg 0$ if and only if for every $i \neq j $, $\gcd_n(a_i, a_j)$ stabilizes (notice that if $D_n$ is Noetherian and $\gcd_n(a,b)$ stabilizes, then also $\gcd_n(a,b,c)$ stabilizes). 
 
In the next lemma and in its corollary, we derive from this fact a characterization of when a finite intersection of principal ideals of $D$ is principal involving the use of primitive ideals.

\begin{definition} \cite{arnold}
\label{primitive}
 A proper ideal $I$ of an integral domain $D$ is \it primitive \rm if $I$ is not contained in any proper principal ideal of $D$.
 \end{definition}

\begin{lemma} 
\label{primitiv}
 Let $D = \bigcup_{n=1}^{\infty} D_n$ be the directed union of some GCD domains $D_n$. Let $a,b \in D$ let $d_n:=\gcd_n(a,b)$. Then, there exists a positive integer $N$ such that $ \gcd_n(a,b) = \gcd_N(a,b) $ for every $n \geq N$ if and only if the ideal $J_n = \left( \dfrac{a}{d_n}, \dfrac{b}{d_n} \right)D $ is primitive for some $n$.
\end{lemma}

\begin{proof}
First assume that, for every $n$, there exists a non unit $x_n \in D$ such that $J_n \subseteq x_nD$. Thus, since $\gcd_n( \frac{a}{d_n}, \frac{b}{d_n} )= 1$, we have $x_n \not \in D_n$. But $x_n \in D_m$ for some $m >n$, and hence $d_n$ properly divides $d_m$. 

Conversely, assume that $J_n$ is primitive. Since $d_n$ divides $d_m$ in $D$ for every $m > n$, then $J_n \subseteq \frac{d_m}{d_n}D$ because $\frac{a}{d_n}= \frac{d_m}{d_n} \frac{a}{d_m}$ and the same happens for $b$. But $J_n$ is primitive and therefore $d_n$ is associated to $d_m$ for every $m > n.$ 
\end{proof}




\begin{corollary} 
\label{primitiveprop}
Let $D = \bigcup_{n=1}^{\infty} D_n$ be the directed union of some Noetherian GCD domains $D_n$. Pick some elements $a_1, a_2, \ldots, a_n \in D$ and let $d_{n_{ij}}:=\gcd_n(a_i, a_j)$. The ideal $$I= a_1D \cap a_2D \cap \ldots \cap a_sD$$ is principal if and only if for every $i,j$, the ideals $$  \left( \dfrac{a_i}{d_{n_{ij}}}, \dfrac{a_j}{d_{n_{ij}}} \right)D $$ are primitive for some $n$.
\end{corollary}

\begin{lemma} 
\label{lemmax}
Let $\m$ be the maximal ideal of a local domain $D$. If $\m $ it is not a directed union of principal ideals, then there exist $a,b \in \m $ such that the ideal $(a,b)D$ is primitive.
\end{lemma}

\begin{proof}
Set $\m= \sum_{\lambda \in \Lambda} x_{\lambda}D$ and assume that $(x_{\lambda}, x_{\mu})D \subseteq x_{\lambda^{\star}}D \subseteq \m$ for every $ \lambda, \mu \in \Lambda $. We can also assume $ \lambda^{\star} \in \Lambda $ since we do not assume that $ \lbrace x_{\lambda} \rbrace_{\lambda \in \Lambda}  $ is a minimal set of generators for $\m$.  Hence we can write $ \m = \bigcup_{\lambda \in \Lambda}x_{\lambda}^{\star}D $ which is a directed union of principal ideals.
\end{proof}

\medskip

Next theorem generalizes the fact that a quadratic Shannon extension is a GCD domain (or a finite conductor) if and only if it is a valuation domain \cite[Theorem 6.2]{GHOT}. Indeed a quadratic Shannon extension is a local directed union of regular local rings (hence of Noetherian UFDs) and its maximal ideal is a directed union of principal ideals. We are going to make use of the next result also when we will study the GCD property for monoidal Shannon extensions in Theorem \ref{monoidal}. 

\begin{theorem} 
\label{directSBID}
Let $D = \bigcup_{n=1}^{\infty} D_n$ be the directed union of some Noetherian GCD domains $D_n$. Then the following conditions are equivalent:
\begin{enumerate}
\item $D$ is a SBID.
\item $D$ is local and its maximal ideal $ \m $ is a directed union of principal ideals.
\end{enumerate}
\end{theorem}

\begin{proof}
Assume (1). The fact that $D$ is local is proved in Proposition \ref{localSBID}. Let $\m$ be the maximal ideal of $D$ and assume that it is not a directed union of principal ideals. By Lemma \ref{lemmax} this implies that there exist $a,b \in \m $ such that the ideal $(a,b)D$ is primitive. Hence for every $n$ large enough such that $a,b \in D_n$, we have that $(a,b)D_n$ is a proper ideal of $D_n$ and moreover, since $(a,b)D$ is a primitive ideal of $D$, we have that $ d_n:= \gcd_n(a,b) $ is a unit of $D$. Hence $$ aD \cap bD = \bigcup_{n \geq N}( aD_n \cap bD_n) = \bigcup_{n \geq N} \left( \dfrac{ab}{d_n}D_n \right) = abD. $$ This contradicts the fact that $D$ is a SBID since both $a$ and $b$ are non-unit of $D$. 

Conversely, assume (2) and take $I= a_1D \cap a_2D \cap \ldots \cap a_sD$ for $a_1, a_2, \ldots, a_s \in D$. We may assume that $a_i \not \in a_jD$ for every $i \neq j$ and prove that in this case $I$ is not finitely generated. 
By Corollary \ref{primitiveprop}, $I$ it is not finitely generated if there exists $i, j \in \lbrace 1, \ldots, s \rbrace$ such that the ideal $$ J_n= \left( \dfrac{a_i}{d_{n_{ij}}}, \dfrac{a_j}{d_{n_{ij}}} \right)D $$ is not primitive for every $n \gg 0$ (we recall that, taking the same notation of Corollary \ref{primitiveprop}, $d_{n_{ij}}:=\gcd_n(a_i, a_j)$ ). 
By the assumption of $a_i \not \in a_jD$ for every $i \neq j$, we have $ \frac{a_i}{d_n}, \frac{a_j}{d_n} \in \m $. 
Since $\m$ is a directed union of principal ideals, there exists a principal ideal $xD$ of $D$ which contains $ \frac{a_i}{d_n}, \frac{a_j}{d_n} $ and hence $J_n$ is not primitive for every $n$. 
\end{proof}



  
 \medskip

\section{Generalities and examples of directed unions of local monoidal transforms} \label{monsection}

In this section we introduce the directed unions of local monoidal transforms of a regular local ring, calling these rings \it monoidal Shannon extensions\rm. We will recall some results about directed union of quadratic transforms from \cite{HLOST} and use them to compare quadratic Shannon extensions with examples of monoidal Shannon extensions. Regular local rings are well known to be UFD and hence GCD domain, thus directed unions of local monoidal transforms are explicit examples of the rings described in the previous section.

\begin{definition} \label{defmontrasf}
  Let $(R,\m)$ be a regular local ring and let $\p$ be a  regular prime ideal of $R$
  with $\hgt \p \ge 2$.  Let $x \in \p \setminus \p^2$ and let $\m_1$ be a maximal ideal of $R[\frac{\p}{x}]$ such that  $\m_1 \cap R = \m$. Then, the ring $$R_1=R \left[ \frac{\p}{x} \right]_{\m_1}$$ is called a \textit{local monoidal transform} of $R$. If $\p=\m$ is the maximal ideal of $R$, the ring $R_1$ is called a \textit{local quadratic transform} of $R$. 
  \end{definition}

 A useful tool for the investigation on local monoidal transforms is the canonical map $\Spec R_1 \to \Spec R$ sending a prime ideal $Q$ of $R_1$ to its contraction $Q \cap R.$ One property of interest is the biregularity of such map at one prime ideal of $R_1$.
  
  \begin{definition} \label{birdef}
  Given an extension of integral domains $ A \hookrightarrow B $, we say that the map $\Spec B \to \Spec A$ is biregular at a prime ideal $Q \subseteq B$ if $B_Q= A_{(Q \cap A)}$.
  \end{definition} 
  
  We record in Proposition~\ref{blowup} well known properties of $R_1$ and of the 
  map $\Spec R_1 \to \Spec R$. The references for them are \cite[Section 2]{Sha} and \cite[Section 1]{Abh3}.
  
  \begin{proposition} \label{blowup}  Assume notation as above.  Then:
  \begin{enumerate}
  \item  $R_1$ is a regular local ring.
  \item
  $\p R_1 = xR_1$ is a 
  height one prime ideal of $R_1$.  
  \item $xR_1 \cap R = \p$ and $(R_1)_{xR_1}$ is the order valuation ring defined by the 
  powers of $\p$.
  \item
  $R_1[\frac{1}{x}]=R[\frac{1}{x}]$. 
  \item
  The map $\Spec R_1 \to \Spec R$ is  biregular at every prime ideal $\q$ of $R_1$ such that $xR_1 \nsubseteq \q$. In particular, if $\q$ is a height one prime of $R_1$ other than $xR_1$,
  then $\q \cap R $ is a height one prime of $R$, and $R_{\q \cap R} = (R_1)_\q$.
 \end{enumerate}
   \end{proposition}
  
  
  The properties of $R_1$ given by Proposition~\ref{blowup} can be used inductively to study iterated sequences of local monoidal transforms.
  
  \begin{definition} 
Let   $\{(R_n, \m_n) \}_{n \in \N}$ be  an infinite directed sequence of  local monoidal (resp. quadratic) transforms of a regular local ring $R$, 
that is for every $n$,  $$R_{n+1}= R_{n} \left[ \dfrac{\p_n}{x_n} \right] _{\m_{n+1}}, $$ where $\p_n$ is a regular prime ideal (resp. maximal ideal) of the regular local ring $R_n$, $x_n \in \p_n \setminus \p_n^2$ 
and the inclusion map $R_n \hookrightarrow R_{n+1}$ is a  local map (i.e. $ \m_{n} \subseteq \m_{n+1} $). 

Since we are interested in studying the ideal theoretic properties of an infinite union of such rings, we assume all the ring $(R_n, \m_n)$ to have the same dimension $d \geq 2$.

The infinite directed union $S = \bigcup_{n \in \N}R_n$ is a local integrally closed overring of $R$ with maximal ideal $\m_S = \bigcup_{n \in \N} \m_n$. We call  $S$  a {\it monoidal (resp. quadratic)  Shannon extension } of $R$.
\end{definition}


In \cite{HLOST}, the authors call an infinite directed union of local quadratic transforms of a regular local ring $R$ a Shannon extension of $R$. Here we distinguish the family of quadratic Shannon extensions, from the larger family of rings of the monoidal Shannon extensions. In connection to the results proved in Section 2, we observe that for a quadratic Shannon extension, since for every $n$, $\m_nR_{n+1}= x_nR_{n+1}$ is principal by Proposition \ref{blowup}(2), then the maximal ideal $$ \m_S = \bigcup_{n \in \N} \m_n = \bigcup_{n \in \N} x_nS $$ is a directed union of principal ideals. Hence quadratic Shannon extensions are SBID by Theorem \ref{directSBID}.

We also recall that, fixed a valuation overring $V$ of $R$, there exists a unique sequence of iterated local quadratic transforms of $R$ such that  $R_n \subseteq V$ for every $n$ and this sequence is called the \it sequence of local quadratic transforms of $R$ along $V$\rm. Moreover, by \cite[Proposition 4]{Abh}, this sequence is finite if and only if $V$ is a prime divisor of $R$, that is a valuation overring birationally dominating $R$ such that $$ \trdeg \left( \frac{V}{\m_V}, \frac{R}{\m} \right) = \dim R - 1.$$

In \cite[Corollary 3.9 ]{HLOST} it is proved that a quadratic Shannon extension is Noetherian if and only if it is a DVR. We are going to prove that a monoidal Shannon extension of $R$ is Noetherian if and only if it is a regular local ring of dimension less than $\dim R$.

\begin{proposition}\label{mingen}
  Let $S$ be a monoidal Shannon extension of a regular local ring $R$ of dimension $d \geq 2$.
  If $\m_S$ is finitely generated, then it is minimally generated by at most $d - 1$ elements.
  Moreover, any minimal generating set of $\m_S$ is a regular sequence on $S$
  and part of a regular system of parameters of $R_n$ for $n \gg 0$.
\end{proposition}

\begin{proof}
  Take a minimal generating set for the maximal ideal $\m_S$ of $S$, say $\m_S = (x_1, \ldots, x_t)$,
  and take $n \ge 0$ such that $x_1, \ldots, x_t \in R_n$.
  The inclusion $\m_n \subseteq \m_S$ induces a map $$ \frac{\m_n}{\m_n^2} \rightarrow \frac{\m_S}{\m_n \m_S} = \frac{\m_S}{\m_S^2}.$$
  Let $\overline{x_1}, \ldots, \overline{x_t}$ denote the images of $x_1, \ldots, x_t$ in $\frac{\m_n}{\m_n^2}$.
  By Nakayama's Lemma, the images of $\overline{x_1}, \ldots, \overline{x_t}$ in $\frac{\m_S}{\m_S^2}$ are linearly independent over $\frac{S}{\m_S}$,
  hence $\overline{x_1}, \ldots, \overline{x_t}$ are linearly independent over the subfield $\frac{R_n}{\m_n}$.
  We conclude that $x_1, \ldots, x_t$ are part of a regular system of parameters for $R_n$,
  and in particular $t \le d$.
  
  Moreover, $x_1, \ldots, x_t$ is a regular sequence in $S$.
  To see this, let $0 \le k < t$ and let $a x_{k+1} \in (x_1, \ldots, x_k) S$ for some $a \in S$.
  Write $a x_{k+1} = \sum_{i=1}^{k} c_i x_i$ for some $c_1, \ldots, c_k \in S$.
  Take $n \ge 0$ such that $x_1, \ldots, x_{k+1}, c_1, \ldots, c_k, a \in R_n$,
  so that $a x_{k+1} \in (x_1, \ldots, x_k) R_n$.
  Since $x_1, \ldots, x_{k+1}$ are part of a regular system of parameters in $R_n$ and hence a regular sequence on $R_n$,
  it follows that $a \in (x_1, \ldots, x_k) R_n \subseteq (x_1, \ldots, x_k) S$.
  We conclude that $x_1, \ldots, x_t$ is a regular sequence in $S$.
  
  Finally, since $R_{n+1}$ is a localization of a finitely generated $R_n$-algebra and $ \m_n \subseteq \m_{n+1} $, Zariski's Main Theorem \cite[Theorem 4.4.7]{gro} implies that $\hgt \m_n R_{n+1} < d$ for all $n \ge 0$. 
  Hence, since $\dim R_n = d$ for all $n \ge 0$, 
  it follows that $\m_n R_{n+1} \subsetneq \m_{n+1}$ for all $n \ge 0$.
  If $t = d$, then $\m_n = (x_1, \ldots, x_t) R_n$ for $n \gg 0$, contradicting this fact. We conclude that $t \le d - 1$.
\end{proof}

\begin{corollary}\label{noeth}
  Let $S$ be a monoidal Shannon extension of a regular local ring $R$.
  If $S$ is Noetherian, then $S$ is a regular local ring of dimension at most $\dim R - 1$.
\end{corollary}

\begin{proof}
  Since $\m_S$ is finitely generated, Proposition~\ref{mingen} implies that
  $\m_S$ is minimally generated by a regular sequence on $S$, so $S$ is a regular local ring.
  Since $\m_S$ is minimally generated by at most $\dim R - 1$ elements,
  Krull's Altitude Theorem implies that $\dim S \le \dim R - 1$.
\end{proof}

\begin{example}\label{RLR}
  It is possible to construct a Noetherian monoidal Shannon extension of a regular local ring $(R, \m)$ under the assumption of the existence of a DVR $V$ that birationally dominates $R$ but it is not a prime divisor over $R$. 
  For instance, take $R=k[x,y]_{(x,y)}$ where $y \in xk[[x]]$ is a formal power series in $x$ and $y,x$ are algebraically independent over the base field $k$. Then, the DVR $V=k[[x]] \cap k(x,y)$ has residue field $k$ and hence is not a prime divisor over $R$.
  Consider the infinite sequence $(R_n, \m_n)$ of local quadratic transform of $R$ along $V$. By \cite[Proposition 3.4]{GHOT}, $V=\bigcup_{n \in \N}R_n$.
  Take some indeterminates $z_1, \ldots, z_t$ over the quotient field of $R$ and consider the sequence of rings $$ R_n[z_1, \ldots, z_t]_{(\m_n + (z_1, \ldots, z_t))}.   $$ Each ring is a local monoidal transform of the previous ring obtained by setting $\p_n= \m_n$ for every $n$. The directed union of this sequence is the regular local ring $V[z_1, \ldots, z_t]_{(x, z_1, \ldots, z_t)}$.
\end{example}


Now we give a first example of a non-Noetherian monoidal Shannon extension of a $3$-dimensional regular local ring which is neither a quadratic Shannon extension nor a valuation domain. We discuss some properties of this ring. At the end of this article, in Theorem \ref{monoidal} we are going to prove that this ring is a GCD domain showing that monoidal Shannon extensions can be GCD domains without being valuation domains, while this is impossible for quadratic Shannon extensions by \cite[Theorem 6.2]{GHOT}.
  
 \begin{construction} \label{construction}
{\em   Let $(R,\m)$ be a 3-dimensional regular local ring with maximal ideal  $\m = (x, y, z)R$. 
  Define  $$R_1 = R \left[ \frac{y}{x} \right]_{(x, \frac{y}{x}, z)R[ \frac{y}{x}]},$$  and 
  $$R_2 = R_1 \left[ \frac{y}{xz} \right]_{(x, \frac{y}{xz}, z)R_1[ \frac{y}{xz}]}.$$  
  
  Thus $R_1$ is the local monoidal transform of $R$ obtained by blowing up the prime ideal $(x, y)R$, dividing by $x$,  and localizing at the maximal ideal generated by $x, \frac{y}{x}$ and $z$;  and $R_2$ is the local monoidal transform of $R_1$ obtained by blowing up the prime ideal $(\frac{y}{x}, z)R_1$, dividing by $z$,  and localizing at the maximal ideal generated by $x, \frac{y}{xz}$ and $z$.
  
  Define $R_{2n+1}$ and $R_{2n+2}$ inductively so that $R_{2n+1}$   
  is the local 
  monoidal transform of $R_{2n} $ obtained by blowing up the prime ideal
   $(x, \frac{y}{x^nz^n}))R_{2n}$, dividing by $x$,  and 
  localizing at the maximal ideal generated by $x, \frac{y}{x^{n+1}z^n}$ and $z$;  and $R_{2n+2}$ 
  is the local 
   monoidal transform of $R_{2n+1}$ obtained by blowing up the prime ideal 
    $(\frac{y}{x^{n+1}z^n}, z)R_{2n+1}$,
     dividing by $z$,  and 
  localizing at the maximal ideal generated by
   $x, \frac{y}{x^{n+1}z^{n+1}}$ and $z$.  
   Call $S = \bigcup_{n \in \N} R_n$. }
   \end{construction}

 We record  properties of $S$ and of the sequence $\{R_n\}$  in Theorem~\ref{specialcase}. It will be easy to observe that the description of this ring can be generalized in higher dimension following a similar pattern. Among other things, we prove that this ring in Construction \ref{construction} is not Noetherian and it admits a unique minimal Noetherian overring. 
 
 In \cite{HLOST}, the authors proved that a quadratic Shannon $S$ extension that is not a DVR, always admits a unique minimal Noetherian overring and they call it the \textit{Noetherian hull} of $S$. In Theorem 4.1 of \cite{HLOST} are given different characterizations of the Noetherian hull $T$ of $S$. In particular is shown that $T = S [\frac{1}{x}]$ for $x \in \m_S$ such that $xS$ is $\m_S$-primary and that $T$ is a localization of $R_n$ for $n \gg 0$. In particular, $T$ is a Noetherian regular UFD.
 We also describe the complete integral closure of $S$, of which we recall the definition. 
 
 \begin{definition} \label{2.5}
 \rm Let $A$ be an integral domain. An element $x$ in the field of fractions of $A$ is called {\it almost integral} over $A$ if $A [x]$ is contained
  in a principal fractional ideal of $A$. 
  
     The set of all the almost integral elements over $A$ is called the \textit{complete integral closure} of $A$ and it is denoted by $A^{*}$. The ring $A$ is {\it completely integrally closed} if $A=A^*.$
\end{definition}
 
 The complete integral closure of a quadratic Shannon extension is well described in \cite[Section 6]{HLOST}.
 
 \medskip
 
 
 \begin{theorem} \label{specialcase}    Assume notation as in Construction~\ref{construction}.  Let $S = \bigcup_{n \in \N}R_n$, and let $\p = y \mathcal{V} \cap S$ where $\mathcal{V}= R_{yR}$.  
 Then:
 \begin{enumerate}    
 
 \item The maximal ideal of $S $ is $\m_S = (x, z)S$ and $$
 \p  = \bigcap_{n \in \N}x^nS = \bigcap_{n \in \N}z^nS$$ is a non finitely generated prime ideal of $S$.  
 
 \item The 
   principal ideals $xS$ and $zS$ are nonmaximal prime ideals of $S$ of height 2.
 \item 
 $\frac{S}{\p}$ is a 2-dimensional regular local ring that is isomorphic to $\frac{R}{yR}$. 
 This isomorphism  defines 1-to-1 correspondence of 
 the prime ideals of $S$  of height 2  containing $\p$    with the
 prime ideals of $R$ of height 2, containing $y$. 
  \item 
 The localizations $S_{xS}$ and $S_{zS}$ are rank 2 valuation domains, and the map 
  $\Spec S \to \Spec R$ is not biregular at these two prime ideals.
  
 \item 
 Let $\q$ be  a  prime ideal of $S$ of height 1. If $\q \ne \p$, then $\q$ is a principal ideal
 generated by a prime element $f \in R$ such that $f \not\in (x, y) R \cup (y,z)R$.  It follows that  
 $R_{fR} = S_\q$.


 
 \item
$T :=  S[\frac{1}{xz}] $ is  a 
 2-dimensional regular Noetherian UFD and it is 
 the complete integral closure of $S$ and the unique minimal Noetherian overring of $S$.

  \item
  $S = T \cap V_1 \cap V_2$,  where $V_1 = S_{xS}$ and $V_2 = S_{zS}$ are  the valuation rings of item~2.
  
  
 \end{enumerate}  

 \end{theorem}
 
 \begin{proof}
 1) By definition $\frac{y}{x^iz^j} \in \p$ for all $i, j \in \N$.  
 Moreover, all the elements of this form are necessary to generate $\p$. Thus $\p$ is non finitely generated.
 It is also clear that $\p = \bigcap_{n \in \N}x^nS$ and $\p = \bigcap_{n \in \N}z^nS$. Hence $\m_n \subseteq (x,z)S$ for  every $n \in \N $ and this implies that $\m_S = (x,z)S.$

\noindent 2) Observe that, since $x$ and $z$ are prime elements in $R_n$ for every $n \in \N$, then they are prime elements in $S$. Moreover:
 $$
  (S_{xS})_{\p} = (S_{zS})_{\p}  =  S_{\p} = \mathcal V.
 $$
 
 By \cite[Theorem~2.4]{Fon},  it follows that $xS$ and $zS$ are prime ideals of height $2$ and $S_{xS}$  and  $S_{zS}$ are rank 2 valuation domains. The DVR $\mathcal V$ is the  rank 1 valuation overring of both  $S_{xS}$ and  $S_{zS}$.

\noindent    3) Notice that
   \begin{equation*}  \label{eq4}
   \frac{R}{yR} ~ =   \frac{R_1}{(\p \cap R_1)}  = \cdots  = ~\frac{R_n}{(\p \cap R_n)} ~ = 
   \cdots = ~\frac{S}{\p}. 
  \end{equation*}
 Therefore $\frac{S}{\p}$ is a 2-dimensional regular local ring that is isomorphic to $\frac{R}{yR}$.

 \noindent 4)  Since $xS \cap R = (x, y)R$ and $zS \cap R = (y, z)R$ and 
  $R_{(x,y)R}$  and $R_{(y,z)R}$ are  2-dimensional
 regular local rings,  the map $\Spec S \to \Spec R$ is not biregular at $xS$ and $zS$. 

 \noindent 5) Let $\q$ be a height 1 prime ideal of $S$ with $\q \ne \p$. 
 Then $\q \nsubseteq xS \cup zS$ 
 and $xS, zS \nsubseteq \q $. 
 By repeated applications of Proposition~\ref{blowup}(5),
 the map  $\Spec S \to\Spec R$ is biregular at $\q$, that is $S_\q = R_{\q \cap R}$. 
 Since $\dim S_\q = 1$, $\q \cap R = fR$, where $f$ is a prime element of $R$
 with $f \notin (x, y)R \cup (y, z)R$. This implies that $f$ is a prime element in $R_n$ for every $n$ and thus 
  $\q = fS$.  
 
 
\noindent 6)   Iterating Proposition~\ref{blowup}(4), we get that $S[\frac{1}{xz}]= R[\frac{1}{xz}]$.  Hence  $S[\frac{1}{xz}]$  is a 
  regular Noetherian UFD. Since $\frac{R}{yR} = \frac{S}{\p}$, we have $\dim S[\frac{1}{xz}] = 2$.
   Since  $\frac{y}{x^iz^j} \in S$ for all $i, j \in \N$,  $S[\frac{1}{xz}]$ is  almost integral over $S$. 
   It follows that $S^* = S[\frac{1}{xz}]$.
    
 If $A$ is a Noetherian overring of $S$, then $\frac{1}{xz}$ is almost integral and therefore integral over 
  $A$.  Since $ xz \in S \subseteq A$,  it follows that $\frac{1}{xz} \in A$, cf. \cite[page~10, Theorem~15]{Kap}.  We conclude that  $S[\frac{1}{xz}] = T$ 
  is  the Noetherian hull of $S$. 
  
\noindent  7)  Let $$\frac{a}{x^iz^j} \in S \left[ \frac{1}{xz} \right] \cap S_{xS} \cap S_{zS},$$  where $a \in S$ and 
   $i \ge 0$ and $j\ge 0$ are minimal for such a representation.  Then since $a \in S \setminus (xS \cup zS)$ 
   and $xS$ and $zS$ are distinct nonzero principal prime ideals, we get $i = 0 = j$. Therefore $S = T \cap V_1 \cap V_2$.
   \end{proof}


\medskip
 
 The description of property 7 of Theorem \ref{specialcase} is motivated by Theorem 5.4 of \cite{HLOST}, in which it is proved a quadratic Shannon extension is the intersection of its Boundary valuation ring and of its Noetherian hull.  
 
   We recall the definition of Boundary valuation ring of a quadratic Shannon extension:  
   
   \begin{definition} \cite[Corollary 5.3 ]{HLOST} \label{boundary}
  Let $S  = \bigcup_{i \ge 0}R_i$ be a  quadratic Shannon extension of $R = R_0$ and for each $i$, let $V_i$ be the  DVR  
defined by  the {\it order function}  $\ord_{R_i}$,  where for $x \in R_i$, 
$$ \ord_{R_i}(x) = \sup \{ n \mid x \in \m_i^n \} $$ and $\ord_{R_i}$ is extended to the quotient field of $R_i$ by defining $\ord_{R_i}(\frac{x}{y}) = \ord_{R_i}(x) -\ord_{R_i}(y)$ for all $x,y \in R_i$ with $y \ne 0$. 
The family $\{V_i\}_{i=0}^\infty$ determines 
the set  
\begin{equation*}\label{equation V}
	V ~   =   ~    \bigcup_{n \ge 0}~ \bigcap_{i \ge n} V_i = \{ a \in F ~ | ~  \ord_{R_i }(a) \ge 0 \text{ for } i \gg 0 \}.\end{equation*}

The set $V$  consists of the elements in $F$ that are in all but finitely many of the  $V_i$.
In \cite[Corollary 5.3]{HLOST},  is proved that $V$ is a valuation domain that birationally
dominates $S$,  and $V$ is called the {\it Boundary valuation ring} of the Shannon extension $S$.
  \end{definition}
  
  Hence, the rings $V_1 = S_{xS}$ and $V_2 = S_{zS}$ are playing for the monoidal Shannon extension in Construction \ref{construction}, the role that is played by the boundary valuation ring $V$ in the case where $S$ is a quadratic Shannon extension.
   
   If instead we consider the limit point $V$ of the order valuation rings $\{V_n\}$ of the sequence $\{R_n\}$ given in Construction \ref{construction}, we do not find the equality $ S= T \cap V $. To see this, we can take for instance the element $\frac{x}{z} $ which is in $T$ and in $V_n$ for all $n$ but it is not in $S$.

 \begin{remark}     The monoidal Shannon extension $S$ obtained in Construction \ref{construction} is a pullback of the canonical homomorphism 
     $S \to S/\p$ with respect to the
     canonical injection $S \hookrightarrow  T = S[\frac{1}{xz}]$ as in the following diagram  
     
     \begin{center}
\begin{tikzcd}
   S \arrow[rightarrow]{r}\arrow[hookrightarrow]{d} 
  &  \dfrac{S}{\p}\arrow[hookrightarrow]{d}
  \\ 
   T\arrow{r}
 &  \dfrac{T}{\p}
\end{tikzcd}
\end{center}         
In the terminology of Houston and Taylor in \cite{Houston-Taylor}, $S$ is a pullback diagram of type $\square$. This differs from a pullback of type $\square^{*}$ in the fact that the integral domain $T/\p$ in the lower right of the diagram is not a field. Another proof that this specific ring $S$ is a GCD domain can be derived using transfer properties in pullback diagrams \cite[Theorem 5.11]{Houston-Taylor}. 
\end{remark}

\medskip

\section{Chain-prime ideals of monoidal Shannon extensions.}

The example obtained in Construction \ref{construction} can be seen as a prototype to study the structure of a class of monoidal Shannon extensions. In that particular sequence of local monoidal transforms we noticed that the prime ideals $\p_n$ form two different chains, which are $$ \p_0 \subseteq \p_2 \subseteq \p_4 \subseteq \dots \subseteq \bigcup_{k \geq 0}^{\infty} \p_{2k}= xS $$ and $$ \p_1 \subseteq \p_3 \subseteq \p_5 \subseteq \dots \subseteq \bigcup_{k \geq 0}^{\infty} \p_{2k+1}= zS. $$
The directed union of each chain is a prime ideal of the union ring $S$. This is a general fact for a monoidal Shannon extension.

The terminology of the following definition is inspired by the concept of "fundamental locus" of a birational transformation used by Zariski in \cite{ZarF}.

\begin{definition} 
\label{locus}
We call the prime ideals $\p_n$, the \textit{locus ideals} of the sequence $\{(R_n, \m_n) \}_{n \in \N}$. 
\end{definition}

\begin{definition} 
\label{chain}
Let $\{\p_{n_i} \}_{i \in \N}$ a family of locus ideals of the sequence $\{R_n \}_{n \in \N}$. The family $\{\p_{n_i} \}_{i \in \N}$ is a \textit{chain} if $ \p_{n_i} \subseteq \p_{n_{i+1}} $ for every $ i $. 
\end{definition}

We recall, that by Proposition \ref{blowup}(2), $ \p_nS= x_nS $ for some $x_n \in \p_n$.

\begin{proposition} \label{singlechain}  
Let $\{\p_{n_i} \}_{i \in \N}$ be an infinite chain of locus ideals of the sequence $\{R_n \}_{n \in \N}$ and for every $n$, take $x_n \in \p_n$ such that $ \p_nS= x_nS $. Denote $Q= \bigcup_{i \in \N} \p_{n_i}$. Then:
\begin{enumerate}
\item
 $Q$ is a prime ideal of $S$ and it is maximal if and only if $Q \cap R_{n_i}= \m_{n_i} $, for infinitely many $i$.
 \item $Q= \bigcup_{i \in \N} x_{n_i}S$ is the directed union of an ascending chain of principal ideals of $S$.
 \item
 $Q$ is either principal or it is not finitely generated. 
 \end{enumerate}
\end{proposition}

\begin{proof}
 1) Take $ab \in Q= \bigcup_{i \in \N} \p_{n_i}$. Then there exists $n_i$ such that $ab \in \p_{n_i}$ and hence either $a$ or $b$ is in $ \p_{n_i} \subseteq Q $ and hence $Q$ is prime. 
 Assume $ Q \cap R_{n_i} = \m_{n_i}$ for infinitely many $i$. Since for every $n$, $ \m_n \subseteq \m_{n+1} $, we have $Q=\bigcup_{n \geq 0}^{\infty} \m_n= \m_S$. Conversely if $Q$ is maximal, $ Q \cap R_{n_i} = \m_{n_i}$ for every $i$.

 \noindent 2) Since $ \p_{n_i}S = x_{n_i}S $ for some $x_{n_i} \in \p_{n_i}$, we have that $Q= \bigcup_{i \in \N} \p_{n_i} = \bigcup_{i \in \N} x_{n_i}S $ is an ascending union of principal ideals. 
 
 \noindent 3) It follows easily from 2.
 \end{proof} 

\begin{definition} \label{chainprimeideal}
Let $\{\p_{n_i} \}_{i \in \N}$ an infinite chain of locus ideals of the sequence $\{R_n \}_{n \in \N}$.  We call the prime ideal $Q= \bigcup_{i \in \N} \p_{n_i}$ a \it{chain-prime ideal}.
\end{definition}

By Theorem \ref{directSBID} a monoidal Shannon extension $S$ is a SBID if and only if its maximal ideal $\m_S$ is a chain-prime ideal.

The height of a chain-prime ideal of $S$ in general depends on the height of the locus ideal $\p_n$ in the rings $R_n.$ We can define different classes of monoidal Shannon extensions, depending on the height of the locus ideals in the following way:

\begin{definition} \label{MiR} Let $R$ be a regular local ring of dimension $d \geq 2$. For $2 \leq i \leq d$, we call $ \mathcal{M}_i(R) $ the set of the monoidal Shannon extensions of $R$ such that ht $ \p_n \geq i$ for every $n$. 
 \end{definition}
 
The most general setting for monoidal transform is when ht $\p_n= 2$ for every $n$. These transforms have been called by Shannon {\it elementary monoidal transforms}. Shannon proved that any monoidal transform can be factorized in elementary monoidal tranforms (see \cite[Remark 2.5]{Sha}). \\
  It is easy to extend this result in order to prove that there is an inclusion $$ \mathcal{M}_i(R) \subseteq \mathcal{M}_{i-1}(R). $$ With this notation, the set of quadratic Shannon extensions is $\mathcal{M}_d(R)$ while the set of the all possible monoidal Shannon extensions is $ \mathcal{M}_2(R) $.

\begin{proposition} \label{height} Let $R$ be a regular local ring of dimension $d \geq 2.$ 
Let $S \in \mathcal{M}_{d-i}(R)$ be a monoidal Shannon extension of $R$ and let $Q$ be a prime ideal of $S$ which contains infinitely many locus ideal $\p_n$. Then, $$ \hgt Q \geq \dim S - i. $$ 
\end{proposition}
\begin{proof}
 Assume there is an ascending chain of prime ideals of $S$, $$ Q \subsetneq Q_1 \subsetneq \ldots \subsetneq Q_i \subsetneq \m_S, $$ of lenght $i+1$ between $Q$ and $\m_S$. Since all the inclusions are strict, we can find a sufficiently large $n$ such that $$ \p_n \subseteq Q \cap R_n \subsetneq Q_1 \cap R_n \subsetneq \ldots \subsetneq Q_i \cap R_n \subsetneq \m_S \cap R_n = \m_n.$$ But this contradicts the assumption of $S \in \mathcal{M}_{d-i}(R)$, since in that case $\hgt \p_n \geq d-i$.
 \end{proof} 

It follows that, if for instance $\dim R = 3$, then any chain-prime ideal of a monoidal Shannon extension $S$ of $R$ has height at least equal to $\dim S - 1.$ As application of Proposition \ref{height}, we can bound the dimension of the valuation rings belonging to each set $ \mathcal{M}_i(R) $.

Before to do this, we apply by induction Proposition \ref{blowup} in order to study the biregularity of the map $\Spec S \to \Spec R$ (see Definition \ref{birdef}). 

\begin{lemma} \label{biregularity}   
Let $S$ be a monoidal Shannon extension of $R$. 
 Let $P$ be a prime ideal of $S$ which contains only finitely many locus ideals $\p_n$.
 Then, there exists $N \geq 0$, which depends on $P$, such that the map $\Spec S \to \Spec R_N$ is  biregular at $P$. 
\end{lemma}
 
 \begin{proof}
Take $N$ such that $\p_n \nsubseteq P \cap R_{n}$ for all $n \geq N$. Applying Proposition \ref{blowup}(5) with an inductive argument, we get $ (R_N)_{P \cap R_N} = (R_n)_{P \cap R_n} $ for all $n \geq N$ and this gives $S_P= (R_N)_{P \cap R}$. 
 \end{proof}

\begin{theorem} \label{valuations}  
Let $V \in \mathcal{M}_{d-i}(R)$ be a valuation ring which is a monoidal Shannon extension of $R$. Then, $$ \dim V \leq i+2 $$ 
\end{theorem}
\begin{proof}
 When $\dim V=1$ the result is clear, since $i \geq 0.$
 
 Thus assume $\dim V \geq 2$ and call $Q$ the height two prime ideal of $V$. The ring $V_Q$ is a two dimensional valuation domain and therefore it is not a localization of $R_n$ for any $n$. Hence by Proposition \ref{biregularity}, $Q$ must contain infinitely many locus ideals $\p_n$. By Proposition \ref{height}, $ 2 = \hgt Q \geq \dim V - i $, and hence $ \dim V \leq i+2 $.
 \end{proof} 
 
This theorem generalizes the fact that a valuation ring, which is a quadratic Shannon extension of a regular local ring, has dimension at most two \cite[Theorem 8.1]{HLOST}.

\medskip

\section{Monoidal Shannon extensions with principal maximal ideal} \label{2classection}

In this section we describe the family of monoidal Shannon extensions with principal maximal ideal, focusing on the case in which they are in the set $ \mathcal{M}_{d-1}(R) $ (Definition \ref{MiR}).
We assume as before $ \dim R_n=d \geq 3$ for all $n$. 

\begin{proposition} \label{principalchain}  
Let $S $ be any monoidal Shannon extension of $S$ such that $\m_S = xS$ is principal. Then $\m_S$ is a chain-prime ideal.
\end{proposition}
 
 \begin{proof}
 Consider the collection of locus ideals $ \mathcal{C}=\lbrace \p_{n_i} \, | \, x \in \p_{n_i} \rbrace $. This collection is infinite, since otherwise we would find an $n \gg0$ and $y \in \m_n \setminus xR_n$ such that also $y \in \m_S \setminus xS$. We prove that $ \mathcal{C} $ is a chain, proving that, if $x \in \p_{n_i}$, then $ \p_{n_i}R_{n_i+1} = xR_{n_i+1} $. Indeed, if $  x \in \p_{n_i}R_{n_i+1} = tR_{n_i+1} \subseteq tS $, we have $ t \in \m_S $ and $\frac{x}{t}$ is a unit in $S$. Hence, since $\m_{n_i+1} \subseteq \m_S$, we get $\frac{x}{t}$ is a unit in $R_{n_i+1}$ and the proof is complete.
 \end{proof}


 It is easy to observe that a quadratic Shannon extension of a regular local ring $R$ is a monoidal Shannon extension of $R$ with only one chain-prime ideal equal to the maximal ideal $\m_S$. 
The converse of 
 this fact is not true in general. It is possible to construct many monoidal non quadratic Shannon extension with only one chain-prime ideal equal to $\m_S$. But instead it turns out to be true if we assume $S \in \mathcal{M}_{d-1}(R)$ and $\m_S$ to be principal. To prove this we are going to use the characterization of quadratic Shannon extensions as pullbacks proved in \cite{GHOT}. 

 
 For an extensive study about pullback construction in ring theory see for example \cite{FHP, Gabelli-Houston, GH}. We recall that an integral domain $A$ is {\it archimedean} if $\bigcap_{n>0} a^n A = 0$ for each nonunit $a \in A$.

\begin{theorem} \label{quadraticprincipal}  
 Let $S$ 
 be any monoidal Shannon extension of $S$ such that $\m_S = xS$ is principal. Let $Q = \bigcap^{\infty}_{j\geq 0}x^jS$. The following are equivalent:
 \begin{enumerate}
\item
 $S$ is a quadratic Shannon extension of $R_n$ for some $n$. 
 \item
 The ring $S_Q$ is a localization of $R_n$ at the prime ideal $Q \cap R_n$ for some $n$.
 \item
 There are only finitely many locus ideals $\p_n$ contained in $Q$.
\end{enumerate}
\end{theorem}
 
 \begin{proof}
 First consider the case in which $Q=(0)$ and therefore $S$ is a DVR by \cite{Kap}(Exercise 1.5). Now, since $S$ birationally dominates $R$, $S$ is a quadratic Shannon extension by \cite{GHOT}(Proposition 3.4). Clearly (2) and (3) are true in this case. 
 
 
 Suppose now $Q \neq (0)$. By \cite[Theorem 2.4]{Fon}, $S$ occurs in the pullback diagram 
 \begin{center}
\begin{tikzcd}
   S \arrow[rightarrow]{r}\arrow[hookrightarrow]{d} 
  &  \dfrac{S}{Q}\arrow[hookrightarrow]{d}
  \\ 
   S_Q\arrow{r}
 &  \kappa(Q)
\end{tikzcd}
\end{center}
 where $\frac{S}{Q}$ is a DVR. 
  If $S$ is a quadratic Shannon extension, then 
  (2) follows from \cite[Proposition 3.3]{HLOST}. 
Conversely, if $S_Q$ is equal to $(R_n)_{Q \cap R_n}$, we apply \cite[Theorem 4.8]{GHOT} to say that $S$ is a quadratic Shannon extension of $R_n$, since a DVR has divergent multiplicity sequence with respect to any of its regular local subrings (The definition of multiplicity sequence together with results and examples about its relation with Shannon extensions is given in \cite[Section 3]{GHOT}).


We observe that (3) implies (2) by Lemma \ref{biregularity}.
To conclude we need to show that (1) implies (3). But, if $S$ is a quadratic Shannon extension, its locus ideals $\p_n = \m_n$ form eventually a unique chain whose union is $\m_S$. Since $Q \subsetneq \m_S$, only finitely many of them can be contained in $Q$.
 \end{proof}

 We specialize to the case in which the Shannon extensions are in $\mathcal{M}_{d-1}(R)$. We observe that this case completely describes the monoidal Shannon extensions of a regular local ring of dimension $3$.
 
 \begin{corollary} \label{corprinc1}  
 Let $S \in \mathcal{M}_{d-1}(R)$ be a monoidal non-quadratic Shannon extension of $S$ such that $\m_S = xS$ is principal. Let $Q = \bigcap^{\infty}_{j\geq 0}x^jS$. Then $Q$ is a chain-prime ideal.
\end{corollary}

 \begin{proof}
 By Theorem \ref{quadraticprincipal} the ideal $Q$ contains infinitely many locus ideals $\p_n$. Since $S \in \mathcal{M}_{d-1}(R)$, then for any of these ideals and for $n \gg0$, we have $\hgt \p_{n}= d-1$ and therefore, by Proposition \ref{singlechain} and since $Q \subsetneq \m_S$, $\p_n = Q \cap R_n$. Thus, if $\p_n, \p_m \subseteq Q$ and $n < m$,
 $$ \p_{n} = Q  \cap R_{n} \subseteq Q  \cap R_{m} = \p_{m} $$ and hence the infinitely many locus ideals contained in $Q$ form eventually a chain. Again by Proposition \ref{singlechain}, the union of such a chain is a prime ideal of $S$, which is of height at least $\dim S-1$ by Proposition \ref{height}. Hence this union is equal to $Q$ and $Q$ is a chain-prime ideal.
 \end{proof}  
 
\begin{proposition} \label{chainprimeloc}  
Let $S \in \mathcal{M}_{d-1}(R)$ and let $Q= \bigcup_{i \in \N} \p_{n_i} \subsetneq \m_S$ a chain-prime ideal of $S$. Then $ S_Q $ is a quadratic Shannon extension of a regular local ring of dimension $d-1$.
\end{proposition}
 
 \begin{proof}
 As in the proof of the preceding Corollary, since $Q \subsetneq \m_S$, 
 we have $Q \cap R_{n_i} = \p_{n_i}$ for $i \gg 0$. Now, since for $n < m$, $ Q \cap R_n \subseteq Q \cap R_m $, then $Q \cap R_n = \p_n$ if and only if $n \in \{n_i \}_{i \in \N}$.
 



Now consider the rings $ R_n^{\prime}= (R_n)_{(QS_Q \cap R_n)} $. Set theoretically we have $ S_Q = \bigcup_{n \geq 0}^{\infty} R_n^{\prime}$ but we need to show that these rings form a sequence of local quadratic transforms of the regular local ring $R_{\p_0}$. 

For any $k \not \in \{n_i \}_{i \in \N}$, we have $ x_kS= \p_kS \nsubseteq Q $ and hence $x_k \not \in Q$. The ring $R_{k+1}$ is a localization of $R_k[\frac{\p_k}{x_k}]$ at a maximal ideal containing $\p_k$, therefore $ R_{k+1}^{\prime}=R_k^{\prime} $.

Hence we can restrict ourselves to consider the directed union $ S_Q = \bigcup_{i \in \N} R_{n_i}^{\prime}$. For any large $i \gg0$, the contraction $Q \cap R_{n_i}$ is a prime ideal strictly contained in the maximal ideal $\m_{n_i}$ and hence $ Q \cap R_{n_i}= \p_{n_i} $. By this fact and since $\p_{n_i}R_{n_i}^{\prime}$ is the maximal ideal of $R_{n_i}^{\prime}$, for every $i$ the ring $R_{n_{i+1}}^{\prime}$ is a local quadratic transform of $R_{n_i}^{\prime}$ and this completes the proof.
\end{proof}

We describe now the monoidal non-quadratic Shannon extension in $\mathcal{M}_{d-1}(R)$ with principal maximal ideal. Let $x \in S$ such that $\m_S = xS$ and let $Q = \bigcap^{\infty}_{j\geq 0}x^jS$. Notice that $Q \neq (0)$ since $S$ cannot be a DVR, because a DVR would be a quadratic extension. Hence $S$ is the pullback of the diagram \begin{center}
\begin{tikzcd}
   S \arrow[rightarrow]{r}\arrow[hookrightarrow]{d} 
  &  \dfrac{S}{Q}\arrow[hookrightarrow]{d}
  \\ 
   S_Q\arrow{r}
 &  \kappa(Q)
\end{tikzcd}
\end{center} where $\frac{S}{Q}$ is a DVR (\cite{Fon}(Theorem 2.4)).


\begin{theorem} 
 \label{carprinc}  
 Let $S \in \mathcal{M}_{d-1}(R)$ be a monoidal, but not quadratic Shannon extension of $S$ such that $\m_S = xS$ is principal and let $Q = \bigcap^{\infty}_{j\geq 0}x^jS$. 
 Then:
 \begin{enumerate}
 \item $S$ has exactly two chain-prime ideals $Q$ and $\m_S$.
 \item  $S_Q$ is a quadratic Shannon extension of a regular local ring of dimension $d-1$. 
 \item The following are equivalent: \\
 (i) $S$ is a valuation domain. \\
 (ii) $S$ is a GCD domain. \\
 (iii) $S_Q$ is a GCD domain.\\
 (iv) $S_Q$ is a valuation domain.
 \item  $S$ admits a unique proper minimal overring $T$ that is equal to the Noetherian hull of $S_Q$.
\item The complete integral closure $S^{\ast}$ of $S$ is equal to the complete integral closure of $S_Q$.
 \end{enumerate}
 
\end{theorem}

 
\begin{proof}
 1) Proposition \ref{principalchain} implies that $\m_S$ is a chain-prime ideal. 
 By Corollary \ref{corprinc1}, $Q$ is a nonzero chain-prime ideal. Moreover, it is the unique prime ideal of $S$ of height dim$(S)-1$. Take another chain-prime ideal $Q_i= \bigcup_{i \in \N} \p_{n_i} $ of $S$. Proposition \ref{height} implies that the height of $Q_i$ is at least dim$(S)-1$, hence either $Q_i=Q$ or $Q_i= \m_S$.
 

 

\noindent 2) It follows from Theorem \ref{chainprimeloc}.

\noindent 3) The implication (i) $ \Longrightarrow $ (ii) $\Longrightarrow$ (3) are clear. 
The implication (iii) $ \Longrightarrow $ (iv) is obtained applying (2) and \cite[Theorem 6.2]{GHOT}. Finally (iv) $ \Longrightarrow $ (i) follows by a well known transfer property in pullback diagrams \cite{Fon}(Theorem 2.4), since $\frac{S}{Q}$ is a DVR.

\noindent 4) Let $T^{\prime}$ be the Noetherian hull of $S_Q$. If $T$ exists, $T \subseteq T^{\prime}$. Take a Noetherian overring $A$ of $S$. For any $y \in Q$, the element $\frac{y}{x^{n}} \in S \subseteq A$ for every $n \geq 0.$ But $A$ is Noetherian and therefore there is $n \geq 1$ such that $\frac{y}{x^{n}} \in (y, \frac{y}{x}, \ldots, \frac{y}{x^{n-1}})A$. Hence, multiplying $x^n$, we get $ y= y(a_n x^n+ a_{n-1} x^{n-1} + \ldots + a_1 x) $ for some $a_i \in A$. It follows that $1= x(\sum^{n}_{i=1} a_i x^{i-1})$ and $x$ is a unit in $A$. Hence $S_Q= S[\frac{1}{x}] \subseteq A$ and $T^{\prime} \subseteq A$. Thus the Noetherian hull of $S$ exists and coincides with $T^{\prime}$.



\noindent 6) The element $\frac{1}{x}$ is almost integral over $S$, hence we have $S_Q= S[\frac{1}{x}] \subseteq S^{\ast} \subseteq (S_Q)^{\ast}$. Assume first $S_Q$ non archimedean. In this case, by \cite{HLOST}(Theorem 6.9), $ (S_Q)^{\ast}= T = S_Q[\frac{1}{y}]$ where $y \in QS_Q=Q \subseteq S$ is such that $QS_Q= \sqrt{yS_Q}$. Hence $ (S_Q)^{\ast}= S[\frac{1}{y}] $ and $\frac{1}{y}$ is almost integral over $S$. It follows that $ (S_Q)^{\ast}= S^{\ast} $. 
Consider now the case in which $S_Q$ is archimedean. Its complete integral closure $(S_Q)^{\ast}$ is equal to $$(QS_Q :_F QS_Q)= (Q :_F Q)$$ where $F$ is the quotient field of $R$ (see \cite{HLOST}(Theorem 6.2)). Since $ (S_Q)^{\ast}= (Q :_F Q) $ is a fractional ideal of $S$ that contains $S^{\ast}$, it follows that $ (S_Q)^{\ast}= S^{\ast} $.
 \end{proof}

 \begin{example} \label{dv3} Let $R$ be a regular local ring of dimension $3$ with
maximal ideal $\m=(x,y,z)R$ and let $ s_k:= \sum_{j=1}^{k}j $. We start taking at the first step the ideal $\p_0= (x,y)R$ and dividing by $x$ and then taking at the second step $\p_1=(\frac{y}{x}, z)R_1$ and dividing by $\frac{y}{x}$. 

We iterate this process defining two chains of locus ideals $\{\p_{2k} \}_{k \in \N}$ and $\{\p_{2k+1} \}_{k \in \N}$ where $\p_{2k}=(x,\frac{y}{x^{k}})$ and $\p_{2k+1}=(\frac{y}{x^{k}}, \frac{zx^{s_k}}{y^k}  )$, assuming hence that $ \p_{2k}S=xS $ and $ \p_{2k+1}S=\frac{y}{x^k}S $ for all $k$.
 In this way we obtain the sequence of rings $\{(R_n, \m_n) \}_{n \in \N}$ where for $k \geq 0$ the maximal ideals are $$\m_{2k}= (x, \frac{y}{x^{k}},  \frac{zx^{s_k}}{y^{k}}) $$ and $$\m_{2k+1}= (x, \frac{y}{x^{k+1}},  \frac{zx^{s_k}}{y^{k}}). $$  
Let $S = \bigcup_{n \in \N}R_n$ be the monoidal Shannon extension of $R$ obtained as directed union of this sequence of rings. The maximal ideal $\m_S$ of $S$ is the principal ideal $xS$. 
Moreover we can note that for every $n \in \N$ the elements $ \frac{y}{x^n} \in S $ and $ \frac{z}{y^n}=  \frac{zx^{s_{n+1}}}{y^{n+1}} \frac{y}{x^{s_{n+1}}} \in S$. Thus we have $z \in Q:=\bigcap_{n \geq 0}y^{n}$ and $y \in P:=\bigcap_{n \geq 0}x^{n}$. 

We claim that $S$ is a discrete valuation ring of rank 3. Indeed, since $S$ is a local domain with principal maximal ideal, the ideal $P$ is prime and $S$ is a valuation ring if and only if his localization $S_P$ is a valuation ring. The ideal $P$ is generated by the family of elements $\{ \frac{y}{x^n}\}_{n \in \N}$ and hence, since $x$ is a unit in $S_P$ we have $PS_P=yS_P$. By the same fact $S_P$ is a valuation ring if and only if $(S_P)_{QS_P}$ is a valuation ring, but $ (S_P)_{QS_P}= R_{zR} $ is a DVR and hence $S$ is a valuation ring.

The union of the ideals of the first chain $\{\p_{2k} \}_{k \in \N}$ is the maximal ideal $\m_S$ and hence this chain is not minimal. Instead the chain $\{\p_{2k+1} \}_{k \in \N}$ is minimal and has as union the ideal $P$.

In this case $S$ admits a Noetherian hull that is the overring $S_Q= S[\frac{1}{x\frac{y}{x}}]= S[\frac{1}{y}]$. 
\end{example}

\medskip

\section{GCD property for monoidal Shannon extensions whose locus ideals form finitely many chains} \label{ultima}

We can determine whether a monoidal Shannon extension is a GCD domain adding the assumption that the locus ideals of this extension form only finitely many chains. First we need a description of some prime elements of a monoidal Shannon extension.

\begin{lemma} \label{primeelements}   
Let $S$ be a monoidal Shannon extension of a regular local ring $R$. Let $y \in R_m$ be a prime element and assume $y \not \in \p_n$ for every $n \geq m$. It follows that $y$ is a prime element of $S$ and the prime ideal $yS$ has height one.
\end{lemma}
 
 \begin{proof}
 We need to prove that $y$ is a prime element in $R_n$ for every $n \geq m$. This implies that $y$ is a prime element of $S$. 
Let $ R_{m+1}= R_m[\frac{\p_m}{t}]_{\m_{m+1}} $. Since $y \not \in \p_m$, it follows $y \not \in tR_{m+1}=\p_mR_{m+1}$. Consider a prime element $f \in R_{m+1}$ such that $y \in fR_{m+1}$. By biregularity (Proposition \ref{blowup}(5)), $fR_{m+1} \cap R_m$ is an height one prime of $R_m$, but, since $y$ is prime in $R_m$, we have $fR_{m+1} \cap R_m= yR_m$. 

Moreover there exists a unit $u \in R_{m+1}$ and an integer $k \geq 0$ such that $$ut^kf \in fR_{m+1} \cap R_m= yR_m$$ and hence $ ut^kf=yc $ for some $c \in R_m$. It follows $f= \frac{yc}{ut^k} $, but $t$ is prime in $R_{m+1}$ and it does not divide $y$. Therefore $ \frac{c}{ut^k} \in R_{m+1} $, $ f \in yR_{m+1} $ and $y$ is prime in $R_{m+1}$. By induction it follows that $y$ is prime in $R_n$ for $n \geq m$ and applying Lemma \ref{biregularity} at $P=yS$, it follows that $yS$ has height one.
 \end{proof}

\begin{definition} \label{finitechains}
We say that a monoidal Shannon extension has finitely many chains of locus ideals if there are only finitely many chain-prime ideals $Q_1, \ldots, Q_c$ and for $n \gg 0$ each locus ideal $\p_n$ is contained in $Q_i$ for some $i=1, \ldots, c.$
\end{definition}

\begin{theorem} 
\label{monoidal}
Let $S = \bigcup_{n=1}^{\infty} R_n$ be a monoidal Shannon extension having finitely many chains with correspondent chain-prime ideals $Q_1, \ldots, Q_c$. Then the following conditions are equivalent:
\begin{enumerate}
\item $S$ is a GCD domain.
\item $S_{Q_i}$ is a valuation domain for every chain-prime ideal $Q_i$.
\end{enumerate}
\end{theorem}

\begin{proof}
Assume $S$ is a GCD domain. Hence $S_{Q_i}$ is a GCD domain for every $i$ since a localization of a GCD domain is again a GCD domain. Moreover, $S_{Q_i}$ is a local directed union of UFDs whose maximal ideal is a directed union of principal ideals. It follows by Theorem \ref{directSBID} that $S_{Q_i}$ is a SBID and therefore it is a valuation domain (for a SBID, GCD domain is equivalent to valuation domain).

Conversely, assume $S_{Q_i}$ to be a valuation domain for every chain-prime ideal $Q_i$ and also assume $Q_i \subsetneq \m_S$ otherwise the result is trivial. 
Let $a,b \in \m_S$. Consider the ideal $$ I= aS \cap bS = \bigcup_{n \geq N}( aR_n \cap bR_n) = \bigcup_{n \geq N} \left( \dfrac{ab}{d_n}R_n \right) $$ where $d_n$ is the great common divisor of $a$ and $b$ in $R_n$ and $N$ is large enough such that $a,b \in R_N.$ By Corollary \ref{primitiveprop}, $I$ is principal if and only if the ideal $$J_n= \left( \dfrac{a}{d_n}, \dfrac{b}{d_n} \right)S $$ is primitive for some $n$. Observe that, for every $n$, $J_n \subseteq J_{n+1}$.

Let $Q$ be a chain-prime ideal of $S$ and assume $(a,b)S \subseteq Q$. Since $Q$ is a directed union of principal ideals of $S$, we have for some $n$, $(a,b)R_n \subseteq qR_n \subseteq Q \cap R_n$. Hence for any large enough $n$, we can write $d_n=q_n s_n$ with $q_n \in Q$ and $s_n \in S \setminus Q$ and therefore $a= q_n s_n \frac{a}{d_n}$ and $b= q_n s_n \frac{b}{d_n}$.

Since $S_{Q}$ is a valuation ring, we may assume $\frac{a}{b} \in S_{Q} = \bigcup_{n=1}^{\infty} (R_n)_{(Q \cap R_n)}$. It follows that for some $N$, we also have $\frac{a}{b} \in (R_N)_{(Q \cap R_N)}$ and hence $ \frac{b}{d_n} \not \in Q $ for every $n \geq N$, otherwise we would have some prime element of $ (R_N)_{(Q \cap R_N)} $ dividing $b$ but not dividing $a$. This implies that $J_n \nsubseteq Q$ for every $n \geq N$. Since there are only finitely many chain-prime ideals, we can find $N$ sufficiently large such that $J_n \nsubseteq Q_i$ for every $i$.

For this reason, eventually replacing $a$ and $b$ with $\frac{a}{d_n}$ and $\frac{b}{d_n}$ for some $n \gg 0$, we may assume $(a,b)S \nsubseteq Q_i$ for every $i$. In this case, if $(a,b) \subseteq yS \subseteq \m_S$ is not primitive, then $(a,b)R_n \subseteq yR_n$ for some $n$. Since $R_n$ is a Noetherian UFD, $(a,b)R_n$ is contained in some principal prime ideals $p_1R_n, \ldots, p_tR_n$ and it is not contained in any other principal prime of $R_n$. The assumption of having $(a,b)S \nsubseteq Q_i$ implies that $p_1R_n, \ldots, p_tR_n$ are not contained in any locus ideal $\p_h$ of $S$. Hence $p_1S, \ldots, p_tS$ are height one prime ideals of $S$ by Lemma \ref{primeelements}. Now $$(a,b)R_n \subseteq J_n \nsubseteq p_jS$$ for every $j=1, \ldots, t$. 

If, by way of contradiction we assume $J_n$ not primitive, by the same argument used above we would have $J_n \subseteq qS$ where $qS$ is an height one principal prime ideal of $S$, $q$ is not contained in any of the locus ideals of $S$ and $q \neq p_j$ for every $j$. Furthermore, since $a,b \in qS$, we can find $m > n$ such that $(a,b)R_n \subseteq qR_m$ and, by Proposition \ref{blowup}(5), we have $$ (R_m)_{qR_m} = (R_n)_{(qR_m \cap R_n)}. $$ Hence, since $qR_m $ is an height one prime of $R_m$, its contraction $qR_m \cap R_n$ is an height one prime ideal of $R_n$ containing $(a,b)R_n$ and different from $p_1R_n, \ldots, p_tR_n$. This is a contradiction and thus $J_n$ has to be primitive and this fact implies that $I$ is principal.
\end{proof}

\medskip

\begin{corollary} 
The monoidal Shannon extension $S$ in Construction \ref{construction} is a GCD domain.
\end{corollary} 

\begin{proof}
Observe that $S$ has two chain-prime ideals that are $xS$ and $zS$ and every locus ideal in contained in one of them. By Theorem \ref{specialcase}, $S_{xS}$ and $ S_{zS} $ are valuation domains. We conclude using Theorem \ref{monoidal}.
\end{proof}
 
 \section*{Acknowledgements}
The author wishes to thank Professors William J. Heinzer, K. Alan Loper, Bruce Olberding and Matthew Toeniskoetter for the very interesting discussions about the contents of this article.


 \end{document}